\documentclass[12pt,leqno]{amsart}
\usepackage{amsmath}
\usepackage{amssymb}

\def\underset#1#2{{\mathrel{\mathop {{}_{} {#2}}\limits_{{#1}_{}}}}}
\def\upplim_#1{\underset{#1}{\overline\lim}\;}
\def\lowlim_#1{\underset{#1}{\underline\lim}\;}
\newtheorem{definition}[equation]{Definition}
\newtheorem{claim}[equation]{\indent{\it Claim}\rm }

\newtheorem{lemma}[equation]{Lemma}

\newtheorem{theorem}[equation]{Theorem}

\newcommand{\C}{{\mathbf{C}}}

\newcommand{\N}{\mathbf{N}}

\renewcommand{\P}{{\mathbf{P}}}

\numberwithin{equation}{section}

\title[Schmidt's subspace theorem for moving hypersurfaces]{Schmidt's subspace theorem for moving hypersurfaces in subgeneral position} 

\author{Si Duc Quang}

\begin{document}

\maketitle 

\begin{abstract}
In this paper, we establish a Schmidt's subspace theorem for moving hypersurfaces in weakly subgeneral position. Our result generalizes the previous results on Schmidt's theorem for the case of moving hypersurfaces. 
\end{abstract}

\def\thefootnote{\empty}

\section{Introduction}
It is known that there is a close relation between Nevanlinna theory and Diophantine Approximation due to the works of Osgood (see \cite{Os1, Os2}) and Vojta (see \cite{V}). Especially, Vojta has given a dictionary which provides the correspondences for the fundamental concepts of these two theories. With this dictionary, one can translate results from Nevanlinna theory into corresponding results from Diophantine Approximation and vice versa. In Nevanlinna theory, the main goal is to establish the second main theorem, which is an inequality where the characteristic functions of meromorphic functions or mappings are bounded above by a sum of certain counting functions. The result in Diophantine Approximation corresponding to the second main theorem is the so-called Schmidt's subspace theorem. Over the last decades much research on these theorems has been done. Many results have been given, for instance \cite{CRY,DT,LG,RS,RV,Shi}.To state some of them, we recall the following.

Let $k$ be a number field. Denote by $M_k$ the set of places (equivalence classes of absolute values) of $k$ and by $M^\infty_k$ the set of archimedean places of $k$. For each $v\in M_k$, we choose the normalized absolute value $|\cdot |_v$ such that $|\cdot |_v=|\cdot|$ on $\mathbf Q$ (the standard absolute value) if $v$ is archimedean, and $|p|_v=p^{-1}$ if $v$ is non-archimedean and lies above the rational prime $p$. For a valuation $v$ of $k$, denote by $k_v$ the completion of $k$ with respect to $v$ and set $$n_v := [k_v :\mathbf Q_v]/[k :\mathbf Q].$$
We put $||x||_v=|x|^{n_v}_v$. The product formula is stated as follows
$$\prod_{v\in M_k}||x||_v=1,\text{ for }  x\in k^*.$$

For $x = (x_0 ,\ldots , x_n)\in k^{n+1}$, define
$$||x||_v :=\max\{||x_0||_v,\ldots,||x_n||_v\},\ v\in M_k.$$
We define the absolute logarithmic height of a projective point $x=(x_0:\cdots :x_n)\in\P^n(k)$ by
$$h(x):=\sum_{v\in M_k}\log ||x||_v.$$
By the product formula, this does not depend on the choice of homogeneous coordinates $(x_0:\cdots :x_n)$ of $x$. If $x\in k^*$, we define the absolute logarithmic height of $x$ by
$$h(x):=\sum_{v\in M_k}\log^+ ||x||_v,$$
where $\log^+a=\log\max\{1,a\}.$

Let $\N_0=\{0,1,2,\ldots\}$ and for a positive integer $d$, we set
$$\mathcal T_d :=\{(i_0,\ldots, i_n)\in\N_0^{n+1}\ :\ i_0+\cdots +i_n=d\}.$$
Let $Q=\sum_{I\in\mathcal T_d}a_Ix^I$ be a homogeneous polynomial of degree $d$ in $k[x_0,\ldots, x_n]$, where $x^I = x^{i_0}\ldots x^{i_n}_n$ for $x=(x_0,\ldots, x_n)$  and $I = (i_0,\ldots,i_n)$. Define $||Q||_v =\max\{||a_I||_v; I\in\mathcal T_d\}$. The height of $Q$ is defined by
$$h(Q)=\sum_{v\in M_k}\log ||Q||_v.$$
For each $v\in M_k$, we define the Weil function $\lambda_{Q,v}$ by
$$\lambda_{Q,v}(x):=\log\frac{||x||_v^d\cdot ||Q||_v}{||Q(x)||_v},\ x\in\P^n(k)\setminus\{Q=0\}.$$ 

Let $\Lambda$ be an infinite index set. Let $A\subset\Lambda$ be an infinite index subset and $a$ a map $A\rightarrow k$.  We denote this map by $(A,a)$.

\begin{definition}\label{def1.1}{\rm
Let $A\subset\Lambda$ be an infinite index subset and $C_1,C_2\subset A$ subsets
of $A$ with finite complement. Two pairs $(C_1, a_1)$ and $(C_2, a_2)$ are said to be  equivalent if there is a subset $C\subset C_1\cap C_2$ such that $C$ has finite complement in $A$ and such that the restrictions of $a_1,a_2$ to $C$ coincide. Denote by $\mathcal R^0_A$ the set of all equivalence classes of pairs $(C, a)$. Then $\mathcal R^0_A$ has a ring structure. Moreover, we can embed $k$ into $\mathcal R^0_A$ as constant functions.}
\end{definition}

In this paper, we regard a collection of points $\{x(\alpha)\in\P^n(k)\ |\ \alpha\in\Lambda\}$ as a map $x:\Lambda\rightarrow\P^n(k)$. We call a moving hypersurface of degree $d$ in $\P^n(k)$ a homogeneous polynomial in $\mathcal R^0_\Lambda [x_0,\ldots, x_n]$ of the form
$$Q(\alpha)=\sum_{I\in\mathcal T_d}a_I(\alpha)x^I,\text{ for all }\alpha\in\Lambda,$$
where $a_I\ (I\in\mathcal T_d)$ are collections of points, not all zero.

Let $Q_1(\alpha),\ldots,Q_q(\alpha),\ (\alpha\in\Lambda)$ be $q$-moving hypersurfaces in $\P^n(k)$ of degrees $d_1,\ldots, d_q$ respectively. Suppose that $Q_i(\alpha)$ is given by 
$$Q_i(\alpha)=\sum_{I^i\in\mathcal T_{d_i}}a_{i,I^i}(\alpha)x^{I^i}.$$

\begin{definition}\label{def1.2}{\rm
An infinite subset $A\subset\Lambda$  is said to be coherent with respect to 
$\{Q_j\}^q_{j=1}$ if for every polynomial $P\in k[...,x_{1,I^1},...,x_{j,I^j},...,x_{q,I^q},...]$, $I^j\in\mathcal T_{d_j}$, which is homogeneous in $(...,x_{j,I^j},...)$ for each $j = 1,\ldots, q$, either $P(...,a_{1,I^1}(\alpha),\ldots,a_{j,I^j}(\alpha),...,a_{q,I^q},...)$ vanishes for all $\alpha\in A$, or it vanishes for only finitely many $\alpha\in A.$}
\end{definition}

\begin{lemma}[{see \cite[Lemma 1.4]{LG} and also \cite[Lemma 2.1]{CRY}}]\label{lem1.3}
There exists an infinite subset $A\in\Lambda$ which is coherent with respect
to $\{Q_j\}^q_{j=1}.$
\end{lemma}

Let $A\subset\Lambda$ be an infinite index subset which is coherent with respect to $\{Q_j\}^q_{j=1}$. If $j\in\{1,\ldots, q\}$ and $I\in\mathcal T_{d_j}$ satisfy $a_{j,I}(\alpha)\ne 0$ for some $\alpha\in  A$, then the set $\{\alpha\in A\ |\ a_{j,I}(\alpha)\ne 0\}$ has finite complement in $A$ by the coherence of $A$. Hence, the pair 
\begin{align*}
\begin{array}{ccc}
\{\alpha\in  A\ |\ a_{j,I}(\alpha)\ne 0\}&\longrightarrow &k,\\
 \alpha&\mapsto &a_{j,J}(\alpha)/a_{j,I}(\alpha)
\end{array}
\end{align*}
belongs to $\mathcal R^0_A$. We denote by $\mathcal R^0_{A,\{Q_i\}_{i=1}^q}$ the subring of $\mathcal R^0_A$ generated over $k$ by all such pairs. Then $\mathcal R^0_{A,\{Q_i\}_{i=1}^q}$ is an entire ring (see \cite{LG}). 

\begin{definition}\label{def1.4}{\rm
We define $\mathcal R_{A,\{Qj\}^q_{j=1}}$ to be the quotient field of $\mathcal R_{A,\{Qj\}^q_{j=1}}^0$.}
\end{definition}
\noindent
We would like to note that if $B\subset  A\subset\Lambda$ are two infinite index subsets and $A$ is coherent with respect to $\{Q_j\}_{j=1}^q$, then so is $B$, and $\mathcal R_{A,\{Qj\}^q_{j=1}}=\mathcal R_{B,\{Qj\}^q_{j=1}}$.

\begin{definition}\label{def1.5}{\rm
Let $x:\Lambda\rightarrow\P^n(k)$ be a collection of points. We say that $x$ is
algebraically non-degenerate (resp. linearly non-degenerate) over $\mathcal R_{\{Qj\}^q_{j=1}}$ if for all infinite subsets $A\subset\Lambda$ that are coherent with respect to $\{Q_j\}^q_{j=1}$, there is no non-zero homogeneous polynomial (resp. non-zero homogeneous polynomial of degree $1$) $Q\in\mathcal R_{A,\{Qj\}^q_{j=1}}[x_0,\ldots,x_n]$ such that $Q(x_0(\alpha),\ldots, x_n(\alpha))=0$, for all $\alpha\in A$ outside a finite subset of $A$.}
\end{definition}

\begin{definition}\label{def1.6}{\rm
We say that a set $\{Q_j\}^q_{j=1}\ (q\ge N+1)$ of homogeneous polynomials
in $\mathcal R^0_\Lambda [x_0,\ldots, x_n]$ is in weakly $N$-subgeneral position if there exists an infinite subset $A\subset\Lambda$ with finite complement such that for any $1\le j_0 <\cdots < j_{N+1}\le q$, and $\alpha\in  A$, the system of equations
\begin{align}\label{1.1}
Q_{j_i}(\alpha)(x_0,\ldots, x_n) = 0, 0\le i \le N,
\end{align}
has only the trivial solution $(x_0,\ldots,x_n)=(0,\ldots,0)$ in $\overline{k}^{n+1}$, where $\overline{k}$ is an algebraic closure of $k$. Furthermore, the set $\{Q_j\}^q_{j=1}\ (q\ge N+1)$ is said to be in $N$-subgeneral position if the equation (\ref{1.1}) has only the trivial solution $(x_0,\ldots,x_n)=(0,\ldots,0)$ in $\overline{k}^{n+1}$ for all $\alpha\in\Lambda$.}
\end{definition}
We will say that the set $\{Q_j\}^q_{j=1}\ (q\ge N+1)$ is in general position (resp. weakly general position) if it is in $n$-general position (resp. weakly $n-$general position).
 
We now state some pairs of corresponding second main theorem in Nevanlinna theory and Schmidt's subspace theorem for the case of moving targets. Firstly, M. Ru - W. Stoll \cite{RS} proved the following second main theorem for moving hyperplanes (this result was also reproved by M. Shirosaki \cite{Shi} with a simpler proof).

\vskip0.2cm
\noindent
\textbf{Theorem A} (cf. \cite{RS,Shi}). {\it Let $\{a_i\}_{i=1}^q$ be hyperplanes in general position in $\P^n(\C)$. Let $f:  {\C}^m \to \P^n(\C)$ be a meromorphic mapping such that $f$ is linearly nondegenerate over $\mathcal R_{\{a_i\}}$. Then for every $\epsilon >0$, we have}
$$||\ \ (q-n-1-\epsilon)T_f(r) \leq \sum_{i=1}^q N_{(f,a_i)}(r)+ o(T_f(r)).$$
Here, by the notation ``$\| P$'' we mean that the assertion $P$ holds for all $r\in [0,\infty)$ excluding a Borel subset $E$ of the interval $[0,\infty)$ with $\int_E dr<\infty$.

Corresponding to this theorem, M. Ru and P. Vojta \cite{RV} gave the following Schmidt's subspace theorem with moving hyperplane targets. 

\vskip0.2cm
\noindent
\textbf{Theorem B} (cf. \cite{RV}). {\it Let $k$ be a number field, $S$ be a finite set of places of $k$, let $q\ge n+1$ be a positive integer and $\epsilon >0$. Let $\Lambda$ be an infinite index set, let $H_1,\ldots, H_q$ be moving hyperplanes in $\P^n(k)$ of degrees $d_1,\ldots, d_q$ respectively and let $x :\Lambda\rightarrow\P^n(k)$ be a collection of points such that:

$\mathrm{(1)}$ $H_1,\ldots, H_q$ is in general position;

$\mathrm{(2)}$ $x$ is linearly non-degenerate over $\mathcal R_{\{H_j\}^q_{j=1}}$;

$\mathrm{(3)}$ $ h(H_j(\alpha)) = o(h(x(\alpha)))\text{ for all } j=1,\ldots, q.$

\noindent
Then there exists an infinite index subset $A\subset\Lambda$ such that
$$\sum_{v\in S}\sum_{j=1}^q\lambda_{H_j(\alpha),v}(x(\alpha))\le ( n+1+\epsilon)h(x(\alpha))$$
for all $\alpha\in  A$.}

\vskip0.2cm 
Later on, G. Dethloff and T. V. Tan extended the second main theorem of M. Ru and W. Stoll to the case of moving hypersurfaces. Following the method of  Dethloff and Tan, L. Giang \cite{LG} and Z. Chen, M. Ru and Q. Yan \cite{CRY} proved a corresponding Schmidt's subspace theorem for moving hypersurfaces in general position. Recently, in \cite{Q16-2}, we extended the result of Dethloff and Tan to the case of moving hypersurfaces in weakly subgeneral position as follows.

\vskip0.2cm 
\noindent
\textbf{Theorem C} (cf. \cite{Q16-2}, see also \cite{Q16-1}). {\it Let $f$ be a nonconstant meromorphic map of $\mathbf{C}^m$ into $\P^n(\mathbf{C})$. Let $\{Q_i\}_{i=1}^q$ be a family of slowly (with respect to $f$) moving hypersurfaces in weakly $N$-subgeneral position with $\deg Q_j = d_j\ (1\le i\le q).$ Assume that $f$ is algebraically nondegenerate over $\mathcal R_{\{Q_i\}_{i=1}^q}$.  Then for any $\epsilon >0$, we have
$$  ||\ (q-(N-n+1)(n+1)-\epsilon)T_f(r)\le \sum_{i=1}^{q}\dfrac{1}{d_i}N^{[L_j]}_{Q_i(f)}(r)+o(T_f(r)),$$
where $L_j$ is a positive number (explicitly estimated).}

\vskip0.2cm
Our main result in this paper is motivated by the relationship between Nevanlinna theory and Diophantine approximation. We will prove a Schmidt's subspace fo moving hypersurfaces in weakly subgeneral position which corresponds to Theorem C. Namely, we will prove the following.

\noindent 
\textbf{Main Theorem.}\ {\it Let $k$ be a number field, $S$ be a finite set of places of $k$, let $q\ge N+1$ be a positive integer and $\epsilon >0$. Let $\Lambda$ be an infinite index set, let $Q_1,\ldots, Q_q$ be moving hypersurfaces in $\P^n(k)$ of degrees $d_1,\ldots, d_q$ respectively and let $x :\Lambda\rightarrow\P^n(k)$ be a collection of points such that:

$\mathrm{(1)}$ the family of polynomials $Q_1,\ldots, Q_q$ is in weakly $N$-subgeneral position;

$\mathrm{(2)}$ $x$ is algebraically non-degenerate over $\mathcal R_{\{Qj\}^q_{j=1}}$;

$\mathrm{(3)}$ $ h(Q_j(\alpha)) = o(h(x(\alpha)))\text{ for all } j=1,\ldots, q.$

\noindent
Then there exists an infinite index subset $A\subset\Lambda$ such that
$$\sum_{v\in S}\sum_{j=1}^q\frac{1}{d_j}\lambda_{Q_j(\alpha),v}(x(\alpha))\le ( (N-n+1)(n+1)+\epsilon)h(x(\alpha))$$
for all $\alpha\in  A$.}

\vskip0.2cm 
We would like to note that, when the family of moving hypersurfaces is in weakly general position, i.e., $N=n$, our above result will imply the Schmidt's subspace theorems of L. Giang and also of Z. Chen, M. Ru and Q. Yan.

\section{Some lemmas and theorems}

We recall some algebraic lemmas from \cite{CZ,DT}

\begin{lemma}[{see \cite[Lemma 2.2]{CZ}}]\label{2.5}
Let $A$ be a commutative ring and let $\{\phi_1,\ldots ,\phi_p\}$ be a regular sequence in $A$, i.e., for $i=1,\ldots ,p,\phi_i$ is not a zero divisor of $A/(\phi_1,\ldots ,\phi_{i-1})$. Denote by $I$ the ideal in $A$ generated by $\phi_1,\ldots ,\phi_p$. Suppose that for some $q,q_1,\ldots ,q_h\in A$, we have an equation
$$\phi_1^{i_1}\cdots\phi_p^{i_p}\cdot q=\sum_{r=1}^h\phi_1^{j_1(r)}\cdots\phi_p^{j_p(r)}\cdot q_r, $$
where $(j_1(r),\ldots ,j_p(r))>(i_1,\ldots ,i_p)$ for $r=1,\ldots ,h$. Then $q\in I$.
\end{lemma}
Here, as throughout this paper, we use the lexicographic order on $\mathbf N_0^p$. Namely, 
$$(j_1,\ldots,j_p)>(i_1,\ldots,i_p)$$
iff for some $s\in\{1,\ldots ,p\}$ we have $j_l=i_l$ for $l<s$ and $j_s>i_s$.

\begin{lemma}[{see \cite[Lemma 3.2]{DT}}]\label{2.6}
Let $\{Q_i\}_{i=1}^q\ (q\ge n+1)$ be a set of homogeneous polynomials of common degree $d\ge 1$ in $\mathcal R^0_\Lambda[x_0,\ldots ,x_n]$ in weakly general position. Then for any pairwise different $1\le j_0,\ldots ,j_n\le q$, the sequence $\{Q_{j_0},\ldots,Q_{j_n}\}$ of elements in $\mathcal K_{\{Q_i\}}[x_0,\ldots,x_n]$ is a regular sequence, as well as all its subsequences.
\end{lemma}

Let $x$ be a map from an index set $\Lambda$ into $\P^n(k)$. A map $(C,a)\in\mathcal R^0_\Lambda$ is said to be small with respect to $x$ if
$$h(a(\alpha)) = o(h(x(\alpha))),$$
which means that, for every $\epsilon>0$, there exists a subset $C'\subset C$ with finite complement such that $h(a(\alpha))\le \epsilon h(x(\alpha))$ for all $\alpha\in C'$. Denote by $\mathcal K_x$ the set of all small maps. Then, $\mathcal K_x$ is a subring of $\mathcal R^0_\Lambda$.

We denote by $\mathcal C_x$ the set of all positive functions $f$ defined over $\Lambda$ outside a finite subset of $\Lambda$ such that
$$\log^+(f(\alpha)) = o(h(x(\alpha))).$$
Then, $\mathcal C_x$ is a ring. We see that if  $(C,a)\in\mathcal K_x$ then for every $v\in M_k$, the function $||a(\alpha)||_v\in\mathcal C_x$. Furthermore, if $(C,a)$ satisfies that $a(\alpha)\ne 0$ for all $\alpha\in C$ outside a finite subset then the function $\frac{1}{||a(\alpha)||_v}$ also belongs to $\mathcal C_x$

The following lemma is from \cite{LG}.
\begin{lemma}[{see \cite[Lemma 2.2]{LG}}]\label{lem2.3}
Let $\{Q_i\}_{i=0}^n$ be a set of homogeneous polynomials of common degree $d$ in $\mathcal R^0_\Lambda[x_0,...,x_n]$ in general position. Let $A$ be coherent with respect to $\{Q_i\}_{i=0}^n$. Assume that all coefficients of $Q_i,\ i\in\{0,...,n\}$ belong to the field $\mathcal R_{A,\{Q_i\}_{i=0}^n}$. Then, for every $v\in S$, there exist functions $l_{1,v}, l_{2,v}$ such that
$$l_{2,v}(\alpha)||x(\alpha)|^d_v\le\max_{0\le i\le n}||Q_i(\alpha)||_v\le l_{1,v}||x(\alpha)||^d_v,$$
for all $\alpha\in A$ outside a finite subset of $A$. Moreover, if the coefficients of $Q_i,\ i=0,...,n$ belong to $\mathcal K_x$ then $l_{1,v},l_{2,v}\in\mathcal C_x$.
\end{lemma}

To prove our main theorem, we need the following result of Ru and Vojta.

\begin{theorem}[{Schmidt's subspace theorem for moving hyperplane targets \cite[Theorem 1.1]{RV})}]\label{thm2.4}
Let $k$ be a number field, $S$ be a finite set of places of $k$, and let $\epsilon >0$. Let $\Lambda$ be an infinite index set. For every $v\in  S$, let $L^{(v)}_1,\ldots, L^{(v)}_q$ be moving hyperplanes $\Lambda\rightarrow \P^n(k)^*$ and let $x:\Lambda\rightarrow\P^n(k)$ be a collection of points such that we have the following:

$\mathrm{(1)}$ For every $v\in  S,\alpha\in\Lambda, L^{(v)}_1(\alpha),\ldots, L^{(v)}_q(\alpha)$ are in general position.

$\mathrm{(2)}$ For every $v\in  S$, $x$ is non-degenerate over $\mathcal R$ with respect to $L^{(v)}_1,\ldots, L^{(v)}_q$.

$\mathrm{(3)}$ For every $v\in  S, h(L^{(v)}_j(\alpha)) = o(h(x(\alpha))), j=1,\ldots, q.$

\noindent
Then, there exists an infinite index subset $A\subset\Lambda$ such that
$$\sum_{v\in S}\sum_{j=1}^q\lambda_{L_j^{(v)},v}(x(\alpha))\le (n+1+\epsilon)h(x(\alpha)),$$
for all $\alpha\in  A$.
\end{theorem}

We note that M. Ru and P. Vojta only stated their theorem in which the set of hyperplanes is fixed but their proof also valid for the case of moving hyperplanes as above.

\section{Schmidt's subspace theorem for moving hypersurfaces}

We first prove the following lemma.

\begin{lemma}\label{lem3.1}
Let $Q_1,...,Q_{N+1}$ homogeneous polynomials in $\mathcal R^0_\Lambda[x_0,...,x_n]$ of the same degree $d\ge 1$, in weakly $N$-subgeneral position. Then for an infinite subset $A\subset\Lambda$ which is coherent with respect to $\{Q_1,...,Q_{N+1}\}$, there exist $n$ homogeneous polynomials $P_{2},...,P_{n+1}$ in $\mathcal R^0_\Lambda[x_0,...,x_n]$ of the form
$$P_t=\sum_{j=2}^{N-n+t}c_{tj}Q_j,\ c_{tj}\in k,\ t=2,...,n+1,$$
such that $\{P_1,...,P_{n+1}\}$ are in weakly general position in $\mathcal R^0_A[x_0,...,x_n]$, where $P_1=Q_1$.
\end{lemma}
\begin{proof} We assume that $Q_i\ (1\le i\le N+1)$ has the following form
$$ Q_i=\sum_{I\in\mathcal T_d}a_{i,I}x^I.$$
By the definition of weakly subgeneral position, there exists an infinite subset $A'$ of $A$ with finite complement such that, for all $\alpha\in A'$, the system of equations 
$$ Q_i(\alpha)(x_0,...,x_n)=0, 1\le i\le N+1, $$
has only the trivial solution $(0,...,0)$ in $\bar k$. We note that $A'$ is coherent with respect to $\{Q_i\}_{i=1}^{N+1}$. Replacing $A$ by $A'$ if necessary, we may assume that $A'=A$. 

We fix $\alpha_0\in A$. For each homogeneous polynomial $Q\in\ k[x_0,...,x_n]$, we will denote by $Q^*$ the hypersurface in $\P^n(\bar k)$ defined by $Q$, i.e.,
$$ Q^*=\{(x_0:\cdots :x_n)\in\P^n(\bar k)\ |\ Q(x_0,...,x_n)=0\}.$$
Set $P_1=Q_1$. It is easy to see that
$$\dim\left(\bigcap_{i=1}^tQ^*_i(\alpha_0)\right)\le N-t\text{ for } t=N-n+2,...,N+1,$$
where $\dim\emptyset =-\infty$.

The homogeneous polynomial $P_2$ is constructed as follows. For every irreducible component $\Gamma$ of dimension $n-1$ of $Q^*_1(\alpha_0)$, we put 
$$V_{1\Gamma}=\left\{c=(c_2,...,c_{N-n+2})\in k^{N-n+1}\ ;\ \Gamma\subset Q^*_c(\alpha_0),\text{ where }Q_c=\sum_{j=2}^{N-n+2}c_jQ_j\right\}.$$
Then $V_{1\Gamma}$ is a subspace of $k^{N-n+1}$. Since $\dim\left(\bigcap_{i=1}^{N-n+2}Q^*_i(\alpha_0)\right)\le n-2$, there exists $i\in\{2,...,N-n+2\}$ such that $\Gamma\not\subset Q^*_i(\alpha_0)$. This implies that $V_{1\Gamma}$ is a proper subspace of $k^{N-n+1}$. Since the set of irreducible components of dimension $n-1$ of $Q^*_1(\alpha_0)$ is finite, 
$$k^{N-n+1}\setminus\bigcup_{\Gamma}V_{1\Gamma}\ne\emptyset. $$
Hence, there exists $(c_{12},...,c_{1(N-n+2)})\in k^{N-n+1}$ such that
$$\Gamma\not\subset P^*_2(\alpha_0)$$
for all irreducible components of dimension $n-1$ of $Q^*_1(\alpha_0)$, where
$P_2=\sum_{j=2}^{N-n+2}c_{1j}Q_j.$
This clearly implies that $\dim\left(P_1^*(\alpha_0)\cap P_2^*(\alpha_0)\right)\le n-2.$

Similarly, for every irreducible component $\Gamma'$ of dimension $n-2$ of $\left(P_1^*(\alpha_0)\cap P_2^*(\alpha_0)\right)$, put 
$$V_{2\Gamma'}=\{c=(c_2,...,c_{N-n+3})\in k^{N-n+2}\ ;\ \Gamma\subset Q^*_c(\alpha_0),\text{ where }Q_c=\sum_{j=2}^{N-n+3}c_jQ_j\}.$$
Hence, $V_{2\Gamma'}$ is a subspace of $k^{N-n+2}$. Since $\dim\left(\bigcap_{i=1}^{N-n+3}Q_i^*(\alpha_0)\right)\le n-3$, there exists $i, (2\le i\le N-n+3)$ such that $\Gamma'\not\subset Q_i^*(z_0)$. This implies that $V_{2\Gamma'}$ is a proper subspace of $k^{N-n+2}$. Since the set of irreducible components of dimension $n-2$ of $\left(P_1^*(\alpha_0)\cap P_2^*(\alpha_0)\right)$ is finite, 
$$k^{N-n+2}\setminus\bigcup_{\Gamma'}V_{2\Gamma'}\ne\emptyset. $$
Then, there exists $(c_{22},...,c_{2(N-n+3)})\in k^{N-n+2}$ such that
$$\Gamma'\not\subset P_3^*(\alpha_0) $$
for all irreducible components of dimension $n-2$ of $P_1^*(\alpha_0)\cap P_2^*(\alpha_0)$, where
$P_3=\sum_{j=2}^{N-n+3}c_{2j}Q_j.$
It is clear that $\dim\left(P_1^*(\alpha_0)\cap P_2^*(\alpha_0)\cap P_3^*(\alpha_0)\right)\le n-3.$

Repeating again the above step, after the $n$-th step we get the hypersurfaces $P_2,...,P_{n+1}$ satisfying 
$$\dim\left(\bigcap_{j=1}^tP_j^*(\alpha_0)\right)\le n-t, \ 1\le t\le n+1. $$
In particular, $\bigcap_{j=1}^{n+1}P_j^*(\alpha_0)=\emptyset.$

Now we denote by $R$ the resultant of $P_1,...,P_{n+1}$. Then we see that
$$ R(\alpha_0)\ne 0. $$
This implies that $R(\alpha)\ne 0$ for all $\alpha\in A$ outside a finite set. Therefore, the system of equations
$$ P_i(\alpha)(x_0,...,x_n)=0,\ \ i=1,...,n+1, $$
has only the trivial solution $(0,...,0)$ in $\bar k^{n+1}$ for all $\alpha\in A$ outside a finite set. This means that $P_1,...,P_{n+1}$ are in weakly general position in $\mathcal R^0_A[x_0,...,x_n]$. We complete the proof of the lemma.
\end{proof}

We need the following lemma.
\begin{lemma}[{see \cite[Proposition 2.5]{LG} and also \cite{DT}}]\label{3.3}
Let $\{P_i\}_{i=1}^q\ (q\ge n+1)$ be a set of homogeneous polynomials of common degree $d\ge 1$ in $\mathcal R^0_\Lambda[x_0,\ldots ,x_n]$ in weakly general position. Then for any nonnegative integer $L$ and for any $J:=\{j_1,\ldots ,j_n\}\subset\{1,\ldots ,q\},$ the dimension of the vector space $\frac{V_L}{(P_{j_1},\ldots ,P_{j_n})\cap V_L}$ is equal to the number of $n$-tuples $(s_1,\ldots ,s_n)\in\mathbf N^n_0$ such that $s_1+\cdots +s_n\le L$ and $0\le s_1,\ldots,s_n\le d-1 $. In particular, for all $L\ge n(d-1)$, we have
$$\dim\frac{V_L}{(P_{j_1},\ldots ,P_{j_n})\cap V_L}=d^n. $$
\end{lemma}

In order to prove the main theorem, we first prove the following theorem, and then we will show that the main theorem is an implication of it.

\begin{theorem}\label{thm3.2}
Let $k$ be a number field, let $S$ be a finite set of places of $k$ and let $q\ge n+1$ be a positive integer and $\epsilon >0$. Further, let $\Lambda$ be an infinite index set, let $Q_1,...,Q_q$ be moving hypersurfaces in $\P^n(k)$ of the same degree $d$ and let $x: \Lambda\rightarrow\P^n(k)$ be a collection of points such that:

$\mathrm{(1)}$ the family of polynomials $Q_1,...,Q_q$ is in weakly $N$-subgeneral position, $N\ge n$,

$\mathrm{(2)}$ $x$ is algebraically nondegenerate over $\mathcal R_{\{Q_i\}_{i=1}^q}$,

$\mathrm{(3)}$ $h(Q_i(\alpha)) = o(h(x(\alpha)))$ for all $i=1,...,q$.

\noindent
Let $A\subset\Lambda$ be an infinite subset which is coherent with respect to $\{Q_i\}_{i=1}^q$. Suppose that:

$\mathrm{(4)}$ each $Q_i$ has coefficients in $\mathcal R_{A,\{Q_i\}_{i=1}^q}$ and has at least one coefficient equal to $1$,

$\mathrm{(5)}$ $Q_i(\alpha)(x(\alpha))\ne 0$ for all $\alpha\in A$, $i=1,...,q.$

\noindent
Then there exists an infinite index subset $B\subset A\subset\Lambda$ such that
$$\sum_{v\in S}\sum_{i=1}^q\frac{1}{d}\lambda_{Q_i(\alpha),v}(x(\alpha))\le ((N-n+1)(n+1)+\epsilon) h(x(\alpha)),$$
for all $\alpha\in B$.
\end{theorem}
\begin{proof}

We set 
$$\mathcal I=\{(i_1,...,i_{N+1})\ ; 1\le i_j\le q\text{ for } j=1,\ldots,q, i_1,\ldots,i_{N+1}\text{ distinct}\}. $$
For each $I=(i_1,...,i_{N+1})\in\mathcal I$, we denote by $P_{I,1},...,P_{I,n+1}$ the moving hypersurfaces obtained in Lemma \ref{lem3.1} with respect to the family of moving hypersurfaces $\{Q_{i_1},...,Q_{i_{N+1}}\}$. It is easy to see that, for each $v\in S$ there exists a positive function $h_v\in\mathcal C_x$ such that
\begin{align}\label{3.1}
||P_{I,t}(\alpha)(x(\alpha))||_v\le h_v(\alpha)\max_{1\le j\le N+1-n+t}||Q_{i_j}(\alpha)(x(\alpha))||_v, 
\end{align}
for all $I\in\mathcal I$ and $\alpha\in A$ outside a finite subset. Here the function $h_v$ may be chosen common for all $I\in\mathcal I$.

Replacing $A$ by an infinite subset, we may assume that: (\ref{3.1}) holds for all $\alpha\in A$, and for every $v\in S$ there is a permutation $(i_1(v),...,i_q(v))$ of $(1,...,q)$ such that
$$ ||Q_{i_1(v)}(\alpha)(x(\alpha))||_v\le ||Q_{i_2(v)}(\alpha)(x(\alpha))||_v\le\cdots\le ||Q_{i_q(v)}(\alpha)(x(\alpha))||_v. $$
Let $I(v)=(i_1(v),...,i_{N+1}(v))$. Since $P_{I(v),1},\ldots,P_{I(v),n+1}$ are in weakly general position, there exist functions $g_{0,v},g_v\in\mathcal C_x$, which may be chosen independent of $I$ and $\alpha$, such that
$$ ||x(\alpha)||_v^d\le g_{0,v}(\alpha)\max_{1\le j\le n+1}||P_{I(v),j}(\alpha)(x(\alpha))||_v\le g_v(\alpha)||Q_{i_{N+1}(v)}(\alpha)(x(\alpha))||_v.$$
Therefore, we have
\begin{align*}
\prod_{i=1}^q&\dfrac{||x(\alpha)||_v^d}{||Q_i(\alpha)(x(\alpha))||_v}\le g_v^{q-N}(\alpha)\prod_{j=1}^{N}\dfrac{||x(\alpha)||_v^d}{||Q_{i_j(v)}(\alpha)(x(\alpha))||_v}\\
&\le g_v^{q-N}(\alpha)h_v^{n-1}(\alpha)\dfrac{||x(\alpha)||_v^{Nd}}{\bigl (\prod_{j=2}^{N-n+1}||Q_{i_j(v)}(\alpha)(x(\alpha))||_v\bigl )\cdot\prod_{j=1}^{n}||P_{I(v),j}(\alpha)(x(\alpha))||_v}\\
&\le g_v^{q-N}(\alpha)h_v^{n-1}(\alpha)\dfrac{||x(\alpha)||_v^{Nd}}{||P_{I(v),1}(\alpha)(x(\alpha))||_v^{N-n+1}\cdot\prod_{j=2}^{n}||P_{I(v),j}(\alpha)(x(\alpha))||_v}\\
&\le g_v^{q-N}(\alpha)h_v^{n-1}(\alpha)\zeta_v^{(N-n)(n-1)}(\alpha)\dfrac{||x(\alpha)||_v^{Nd+(N-n)(n-1)d}}{\prod_{j=1}^{n}||P_{I(v),j}(\alpha)(x(\alpha))||_v^{N-n+1}},
\end{align*}
where $I(v)=(i_1(v),...,i_{N+1}(v))$ and $\zeta_v$ is a function in $\mathcal C_x$, which is chosen common for all $I(v)\in\mathcal I$, such that 
$$||P_{I(v),j}(\alpha)(x(\alpha))||_v\le\zeta_v(\alpha) ||x(\alpha)||_v^d,\ \forall v\in S$$
and for all $\alpha\in A$ outside a finite subset (then by replacing $A$ by an infinite subset we may assume that this inequality holds for all $\alpha\in A$).

The above inequality implies that
\begin{align}\label{3.2}
\begin{split}
\log\prod_{i=1}^q\dfrac{||x(\alpha)||_v^d}{||Q_i(\alpha)(x(\alpha))||_v}\le&\log(g_v^{q-N}h_v^{n-1}\zeta_v^{(N-n)(n-1)})(\alpha)\\
&+(N-n+1)\log\dfrac{||x(\alpha)||_v^{nd}}{\prod_{j=1}^{n}||P_{I(v),j}(\alpha)(x(\alpha))||_v}.
\end{split}
\end{align}
Now, for each non-negative integer $L$, we denote by $V_L$ the vector space (over $\mathcal R^0_{A,\{Q_i\}}$) consisting of all homogeneous polynomials of degree $L$ in $\mathcal R^0_{A,\{Q_i\}}[x_0,\ldots ,x_n]$ and of the zero polynomial. Denote by $(P_{I,1},\ldots ,P_{I,n})$ the ideal in $\mathcal R^0_{A,\{Q_i\}}[x_0,\ldots ,x_n]$ generated by $P_{I,1},\ldots ,P_{I,n}$.

For each positive integer $L$ divisible by $d$ and for each $(i)=(i_1,\ldots,i_n)\in\mathbf N^n_0$ with $||(i)||=\sum_{s=1}^ni_s\le\frac{L}{d}$, we set
$$W^I_{(i)}=\sum_{(j)=(j_1,\ldots ,j_n)\ge (i)}P_{I,1}^{j_1}\cdots P_{I,n}^{j_n}\cdot V_{L-d||(j)||}. $$
It is clear that $W^I_{(0,\ldots,0)}=V_L$ and $W^I_{(i)}\supset W^I_{(j)}$ if $(i)<(j)$ in the lexicographic ordering. Hence, $W^I_{(i)}$ is a filtration of $V_L$.
  
Let $(i)=(i_1,\ldots ,i_n),(i')=(i_1',\ldots ,i_n')\in\mathbf N^n_0$. Suppose that $(i')$ follows $(i)$ in the lexicographic ordering and $d||(i)||<L$. We consider the following vector space homomorphism
$$\varphi:\gamma\in V_{L-d||(i)||}\mapsto [P_{I,1}^{i_1}\cdots P_{I,n}^{i_n}\gamma]\in\dfrac{W^I_{(i)}}{W^I_{(i')}}, $$
where $[P_{I,1}^{i_1}\cdots P_{I,n}^{i_n}\gamma]$ is the equivalent class in $\frac{W^I_{(i)}}{W^I_{(i')}}$ containing $P_{I,1}^{i_1}\cdots P_{I,n}^{i_n}\gamma$.
We see that $\varphi$ is surjective. 
\begin{claim} $\ker\varphi = (P_{I,1},\ldots ,P_{I,n})\cap V_{L-d||(i)||}$.
\end{claim}
In fact, for any $\gamma\in\ker\varphi$, we have
\begin{align*}
P_{I,1}^{i_1}\cdots P_{I,n}^{i_n}\gamma& =\sum_{(j)=(j_1,\ldots,j_n)\ge (i')}P_{I,1}^{j_1}\cdots P_{I,n}^{j_n}g_{(j)}\ \
&=\sum_{(j)=(j_1,\ldots,j_n)> (i)}P_{I,1}^{j_1}\cdots P_{I,n}^{j_n}g_{(j)},
\end{align*}
where $g_{(j)}\in V_{L-d||(j)||}$. By Lemma \ref{2.5} and Lemma \ref{2.6}, we have $\gamma\in (P_{I,1},\ldots ,P_{I,n})$. Then
$$\ker\varphi\subset (P_{I,1},\ldots ,P_{I,n})\cap V_{L-d||(i)||}. $$
Conversely, for any $\gamma\in (P_{I,1},\ldots ,P_{I,n})\cap V_{L-d||(i)||},\ (\gamma\ne 0)$, we have
$$\gamma =\sum_{s=1}^nP_{I,s} h_s,\ \ h_s\in V_{L-d(||(i)||+1)}. $$
It implies that
$$\varphi (\gamma)=\sum_{s=1}^n[P_{I,1}^{i_1}\cdots P_{I,s-1}^{i_s-1}P_{I,s}^{i_s+1}P_{I,s+1}^{i_s+1}\cdots P_{I,n}^{i_n}h_s]. $$
It is clear that $P_{I,1}^{i_1}\cdots P_{I,s-1}^{i_s-1}P_{I,s}^{i_s+1}P_{I,s+1}^{i_s+1}\cdots P_{I,n}^{i_n}h_s\in W^I_{(i')}$, and hence $\varphi (\gamma)=0$, i.e., $\gamma\in\ker\varphi$. Therefore, we have
$$\ker\varphi = (P_{I,1},\ldots ,P_{I,n})\cap V_{L-d||(i)||}.$$
Hence the claim is proved.

This claim yields that
\begin{align}\label{3.4}
m^I_{(i)}:=\dim\dfrac{W^I_{(i)}}{W^I_{(i')}}=\dim\dfrac{V_{L-d||(i)||}}{(P_{I1},\ldots ,P_{In})\cap V_{L-d||(i)||}}.
\end{align}

Fix a number $L$ large enough (chosen later). Set $u=u_L:=\dim V_L=\binom{L+n}{n}$. We assume that 
$$ V_L=W^I_{(i)_1}\supset W^I_{(i)_2}\supset\cdots\supset W^I_{(i)_K}, $$
where $W^I_{(i)_{s+1}}$ follows $W^I_{(i)_s}$ in the ordering and $(i)_K=(\frac{L}{d},0,\ldots ,0)$. It is easy to see that $K$ is the number of $n$-tuples $(i_1,\ldots,i_n)$ with $i_j\ge 0$ and $i_1+\cdots +i_n\le\frac{L}{d}$. Then we have
$$ K =\binom{\frac{L}{d}+n}{n}.$$
For each $k\in\{1,\ldots ,K-1\}$ we set $m^I_k=\dim\frac{W^I_{(i)_k}}{W^I_{(i)_{k+1}}}$, and set $m^I_K=1$. Then by Lemma \ref{3.3}, $m^I_k$ does not depends on $\{P_{I,1},\ldots ,P_{I,n}\}$ and $k$, but only on $||(i)_k||$. Hence, we set $m_k:=m^I_k$ and $m(l):=m^I_k$, where $l=||(i)_k||$. We also note that by Lemma \ref{3.3}
\begin{align}\label{3.5}
m_k=d^n
\end{align}
for all $k$ with $L-d||(i)_k||\ge n(d-1)$ (it is equivalent to $||(i)_k||\le\dfrac{L}{d}-n$).

From the above filtration, for each $v\in S$, we may choose a basis $\{\psi^{I(v)}_1,\cdots,\psi^{I(v)}_u\}$ of $V_L$ such that  
$$\{\psi^{I(v)}_{u-(m_s+\cdots +m_K)+1},\ldots ,\psi^{I(v)}_u\}$$
 is a basis of $W^{I(v)}_{(i)_s}$. For each $k\in\{1,\ldots,K\}$ and $l\in\{u-(m_k+\cdots +m_k)+1,\ldots, u-(m_{k+1}+\cdots +m_k)\}$, we may write
$$\psi^{I(v)}_l=P_{{I(v)}1}^{i_{1k}}\cdots P_{{I(v)}n}^{i_{nk}}h_{v,l},\ \text{ where } (i_{1k},\ldots,i_{nk})=(i)_k, h_{v,l}\in W^{I(v)}_{L-d||(i)_k||}. $$
We may choose $h_{v,l}$ to be a monomial.

We have the following estimates:
Firstly, we see that
$$\sum_{k=1}^Km_ki_{sk}=\sum_{l=0}^{\frac{L}{d}}\sum_{k:||(i)_k||=l}m(l)i_{sk}=\sum_{l=0}^{\frac{L}{d}}m(l)\sum_{k:||(i)_k||=l}i_{sk}. $$
Note that, by the symmetry $(i_1,\ldots,i_n)\rightarrow (i_{\sigma (1)},\ldots ,i_{\sigma (n)})$ with $\sigma\in S(n)$,  $\sum_{k:||(i)_k||=l}i_{sk}$ does not depend on $s$. We set 
$$ a:=\sum_{k=1}^Km_ki_{sk},\ \text{ which is independent of $s$ and $I$}.$$

Then we have
\begin{align*}
||\psi^{I(v)}_l(\alpha)(x(\alpha))||_v&\le ||P_{I(v),1} (\alpha)(x(\alpha))||_v^{i_{1k}}\cdots ||P_{I(v),n} (\alpha)(x(\alpha))||_v^{i_{nk}}||h_{v,l}(\alpha)(x(\alpha))||_v\\
&\le c_{v,l}(\alpha)||P_{I(v),1}(\alpha)(x(\alpha))||^{i_{1k}}_v\cdots ||P_{I(v),n}(\alpha)(x(\alpha))||_v^{i_{nk}}||x(\alpha)||_v^{L-d||(i)_k||}\\
&=c_{v,l}(\alpha)\left (\dfrac{||P_{I(v),1}(\alpha)(x(\alpha))||_v^{i_{1k}}}{||x(\alpha)||_v^d}\right)^{i_{1k}}\cdots\left (\dfrac{||P_{I(v),n} (\alpha)(x(\alpha))||_v}{||x(\alpha)||_v^d}\right)^{i_{nk}}||x(\alpha)||_v^L,
\end{align*} 
where $c_{v,l}\in\mathcal C_x$. Taking the product on both sides of the above inequalities over all $l$ and then taking logarithms, we obtain
\begin{align}
\nonumber
\log\prod_{l=1}^u||\psi^{(I(v))}_l(\alpha)(x(\alpha))||_v&\le\sum_{k=1}^Km_k\biggl (i_{1k}\log\dfrac{||P_{I(v),1}(\alpha)(x(\alpha))||_v}{||x(\alpha)||_v^d}+\cdots\\
\label{3.6}
&+i_{nk}\log\dfrac{||P_{I(v),n}(\alpha)(x(\alpha))||_v}{||x(\alpha)||_v^d}\biggl)+uL\log ||x(\alpha)||_v+\log c_{v}(\alpha),
\end{align}
where $c_v=\prod_{l=1}^uc_{v,l}\in\mathcal C_x$.

Hence, (\ref{3.6}) gives
\begin{align*}
\log\prod_{l=1}^u||\psi^{I(v)}_l(\alpha)(x(\alpha))||_v&\le a\left(\log\prod_{i=1}^n\dfrac{||P_{I(v),i}(\alpha)(x(\alpha))||_v}{||x(\alpha)||_v^d}\right)+uL\log ||x(\alpha)||_v+\log c_{v}(\alpha),
\end{align*}
i.e.,
\begin{align*}
a\left(\log\prod_{i=1}^n\dfrac{||x(\alpha)||_v^d}{||P_{I(v),i}(\alpha)(x(\alpha))||_v}\right)\le\log\prod_{l=1}^u\frac{||x(\alpha)||_v^L}{||\psi^{I(v)}_l(\alpha)(x(\alpha))||_v}+\log c_v(\alpha),
\end{align*}

Set $c_{0,v}=g_v^{q-N}h_v^{n-1}\zeta_v^{(N-n)(n-1)}(1+c_v^{(N-n+1)/A})\in\mathcal C_x$. Combining the above inequality with (\ref{3.2}), we obtain that
\begin{align}\label{3.8}
\log\prod_{i=1}^q\dfrac{||x(\alpha)||_v^d}{||Q_i(\alpha)(x(\alpha))||_v}\le\frac{N-n+1}{a}\log\prod_{l=1}^u\dfrac{||x(\alpha)||_v^L}{||\psi^{I(v)}_l(\alpha)(x(\alpha))||_v}+\log c_{0,v}(\alpha).
\end{align}

We now fix a basic $\{\phi_1,...,\phi_u\}$ of $V_L$ such that $\phi_i$ is a monomial of the form $x_0^{i_0}...x_n^{i_n}$. Then $\{\psi^{I(v)}_j\}_{j=1}^u$ can be written as independent linear forms $\{L^{I(v)}_j\}$ in $\phi_1,...,\phi_u$, where $\psi^{I(v)}_j(x)=L^{I(v)}_j(F(x))$ with $F(x)=(\phi_1(x):\cdots :\phi_u(x))$.
Then we have
$$ ||\psi^{I(v)}_l(\alpha)(x(\alpha))||_v=||L^{I(v)}_l(\alpha)(x(\alpha))||_v,\ \ 1\le l\le u, \alpha\in A.$$
Hence, we see that
\begin{align*}
\log\prod_{l=1}^u\dfrac{||x(\alpha)||_v^N}{||\psi^{I(v)}_l(\alpha)(x(\alpha))||_v}&=\left (\log\prod_{l=1}^u\frac{||F(x(\alpha))||_v||L^{I(v)}_l(\alpha)||_v}{||L^{I(v)}_l(\alpha)(x(\alpha))||_v}-u\log ||F(x(\alpha))||_v\right )\\
&+uL\log ||x(\alpha)||_v-\log\prod_{l=1}^u||L^{I(v)}_l(\alpha)||_v.
\end{align*}
Combining this inequality with (\ref{3.8}), we get
\begin{align}\nonumber
\frac{a}{N-n+1}&\log\prod_{i=1}^q\dfrac{||x(\alpha)||_v^d||Q_i(\alpha)||_v}{||Q_i(\alpha)(x(\alpha))||_v}\le \biggl (\log\prod_{l=1}^u\frac{||F(x(\alpha))||_v||L^{I(v)}_l(\alpha)||_v}{||L^{I(v)}_l(\alpha)(x(\alpha))||_v}\\
\nonumber
&-u\log ||F(x(\alpha))||_v\biggl )+uL\log ||x(\alpha)||_v-\log\prod_{l=1}^u||L^{I(v)}_l(\alpha)||_v\\
\label{4.11}
&+\frac{a}{N-n+1}\biggl (\log c_{0,v}(\alpha)+\log\prod_{i=1}^q ||Q_i(\alpha)||_v\biggl ).
\end{align}
Since $\phi_i$ is chosen to be a monomial for every $1\le i\le u$, we now write
$$\psi^{I(v)}_l=\sum_{i=1}^uc^{I(v)}_{i,l}\phi_i\in V_L,\ \ c^{I(v)}_{i,l}\in\mathcal R_{A,\{Q_i\}}, $$
for all $l\in\{1,\ldots ,u\}$. This yields that
$$ L^{I(v)}_l(y_1,...,y_u)=\sum_{i=1}^uc^{I(v)}_{i,l}y_i,\ 1\le l\le u. $$

\begin{claim}\label{cl3.1}
For every $\epsilon'>0$, we have
\begin{align*}
\sum_{v\in S}\log\prod_{l=1}^u\frac{||F(x(\alpha))||_v||L^{I(v)}_l(\alpha)||_v}{||L^{I(v)}_l(\alpha)(x(\alpha))||_v}\le (Lu+\epsilon')h(x(\alpha)),
\end{align*}
for all $\alpha$ in an infinite subset of $A$. 
\end{claim}
In deed, we will show that the family of linear forms $\{L^{I(v)l}_l\}_{1\le l\le u}$ will satisfy the assumptions (1), (2) and (3) of Theorem \ref{thm2.4} with respect to the collection of points $F(x): A\rightarrow\P^n(k)$ as follows:

\textbf{Step 1.} We verify the condition (1) of Theorem \ref{thm2.4}. Fix $v\in S$. Since $\{L^{I(v)}_l\}_{1\le l\le u}$ is linear independent over $\mathcal R_{A,\{Q_i\}}$, the determinant $\det (c^{I(v)}_{i,l})_{1\le i,l\le u}\ne 0$ in $\mathcal R_{A,\{Q_i\}}$. Then $\det (c^{I(v)}_{i,l}(\alpha))_{1\le i,l\le u}\ne 0$ for all $\alpha\in A$ outside a finite subset. By passing to an infinite subset of $A$, we may assume that  $\det (c^{I(v)}_{i,l}(\alpha))_{1\le i,l\le u}\ne 0$ for all $\alpha\in A$. This means that $\{L^{I(v)}_l\}_{1\le l\le u}$ is in general position.
 
\textbf{Step 2.} We verify the condition (2) of Theorem \ref{thm2.4}. Suppose to the contrary that $F(x)$ is linear degenerate over $\mathcal R_{\{L^{I(v)}_l\}_{1\le l\le u}}.$ Then there is an infinite subset $B$ of $A$ which is coherent with respect to $\{L^{I(v)}_l\}_{1\le l\le u}$ and a non-zero linear form 
$$L=\sum_{i=1}^ua_iy_i,\ a_i\in\mathcal R_{B,\{L^{I(v)}_l\}_{l=1}^u}$$
 such that
$$ L(\alpha)(F(x(\alpha)))=0 $$
for all $\alpha\in B$ outside a finite subset.  This implies that
$$ Q(\alpha)(F(x(\alpha)))=0, $$
where $Q=\sum_{i=1}^ua_i\phi_i$ is a non-zero homogeneous polynomial in $\mathcal R_{B,\{L^{I(v)}_l\}_{l=1}^u}[x_0,...,x_n]\subset\mathcal R_{B,\{Q_i\}_{i=1}^q}$, for all $\alpha\in B$ outside a finite set. This contradicts the algebraically non-degeneracy of $x$ over $\mathcal R_{\{Q_i\}_{i=1}^q}$.

\textbf{Step 3.} We verify the condition (3) of Theorem \ref{thm2.4}. Since all coefficients of $L^{I(v)}_l\ (1\le l\le u)$ are coefficients of homogeneous polynomials $P^{j_1}_{I(v)1}...P^{j_n}_{I(v)n}$, hence are linear combinations of the coefficients of $Q^{j_1}_{i_1(v)}...Q^{j_{N+1}}_{i_{N+1}(v)}$, there is a positive constant $c_L$ (depending only on $L$) such that
$$ ||L^{I(v)}_l(\alpha)||_\omega=\max_{j}||c^{I(v)}_{j,l}(\alpha)||_\omega\le c_L\max_{1\le j\le n}||Q_{i_j(v)}(\alpha)||_\omega^{N/d}$$
for $\omega\in M_k$. Here, we would like to note that $Q_{i_j(v)}$ has at least one coefficient equal to $1$. Taking the product on both sides of this inequality over all $\omega\in M_k$, we obtain 
\begin{align*}
h(L^{I(v)}_l(\alpha))\le \log c_L+\frac{L}{d}\sum_{i=1}^qh(Q_i(\alpha)).
\end{align*}
We also see that
$$ h(F(x(\alpha)))=Lh(x(\alpha))\text{ and }h(Q_i(\alpha))=o(h(x(\alpha))). $$
Therefore, we have
$$h(L^{I(v)}_l(\alpha))=o(h(x(\alpha)))=o(F(x(\alpha))).$$

Then, by Theorem \ref{thm2.4}, there exists an infinite subset $A'\subset A$ such that
\begin{align*}
\sum_{v\in S}\log\prod_{l=1}^u\frac{||F(x(\alpha))||_v||L^{I(v)}_l(\alpha)||_v}{||L^{I(v)}_l(\alpha)(x(\alpha))||_v}\le (u+\epsilon'/L)h(F(x(\alpha)))=(Lu+\epsilon')h(x(\alpha)),
\end{align*}
for all $\alpha\in A'$. Hence the claim is proved.

Without loss of generality we may assume that $A'=A$. Then we may assume that the claim holds for all $\alpha\in A$.

\begin{claim} For every $v\in S$, we have
\begin{align*}
\log ||L^{I(v)}_l(\alpha)||_v &=o(h(x(\alpha))),\\
\log ||Q_i(\alpha)||_v&=o(h(x(\alpha))),
\end{align*}
for $\alpha\in A$.
\end{claim}
Indeed, we have $\log ||L^{I(v)}_l(\alpha)||_v=\max_{j}\log ||c^{I(v)}_{j,l}(\alpha)||_v$ and for each $c^{I(v)}_{j,l}\not\equiv 0$,
\begin{align*}
\log ||c^{I(v)}_{j,l}(\alpha)||_v&\le\sum_{\omega\in M_k}\log^{+}||c^{I(v)}_{j,l}(\alpha)||_\omega =h(c^{I(v)}_{j,l}(\alpha))=o(h(x(\alpha))),\\ 
\log ||c^{I(v)}_{j,l}(\alpha)||_v&\ge \sum_{\omega\in M_k;||c^{I(v)}_{j,l}(\alpha)||_\omega <1}\log ||c^{I(v)}_{j,l}(\alpha)||_\omega =-h(c^{I(v)}_{j,l}(\alpha))=o(h(x(\alpha))),
\end{align*}
Therefore, we have
\begin{align*}
\log ||L^{I(v)}_l(\alpha)||_v&\le \sum_{\omega\in S; ||c^{I(v)}_{j,l}(\alpha)||_v\ge 1}\log ||c^{I(v)}_{j,l}(\alpha)||_v=o(h(x(\alpha))),\\
\log ||L^{I(v)}_l(\alpha)||_v&\ge \sum_{\omega\in S; ||c^{I(v)}_{j,l}(\alpha)||_v< 1}\log ||c^{I(v)}_{j,l}(\alpha)||_v=o(h(x(\alpha))).
\end{align*}
Thus
$$ \log ||L^{I(v)}_l(\alpha)||_v=o(h(x(\alpha))) $$
with $\alpha\in A$. Similarly, we have
$$ \log ||Q_i(\alpha)||_v=o(h(x(\alpha))) $$
for $\alpha\in A$. Hence, the claim is proved.

Summing both sides of (\ref{3.11}) over all $v\in S$ and using the above two claims with the note that $\log ||F(x(\alpha))||_v=L\log ||x(\alpha)||_v+o(h(x(\alpha)))$, we obtain
\begin{align*}
\frac{a}{N-n+1}\sum_{v\in S}\sum_{i=1}^q\lambda_{Q_i(\alpha),v}(x(\alpha))\le (Lu+\epsilon')h(x(\alpha))+o(h(x(\alpha)))
\end{align*}
with $\alpha\in A$. This is equivalent to 
\begin{align}\label{3.9}
\sum_{v\in S}\sum_{i=1}^q\frac{1}{d}\lambda_{Q_i(\alpha),v}(x(\alpha))\le \frac{N-n+1}{da}(Lu+\epsilon')h(x(\alpha))+o(h(x(\alpha))).
\end{align}

We need some estimate for $a$ and $u$. Firstly, we have
$$ u=\binom{L+n}{n}=\frac{(L +1)\ldots (L+n)}{1\ldots n}=\frac{L^n}{n!}+O(L^{n-1}).$$
For each $I_k=(i_{1k},\ldots ,i_{nk})$ with $||(i)_k||\le \frac{L}{d}-n$, we set 
$$i_{(n+1)k}=\dfrac{L}{d}-n-\sum_{s=1}^ni_s.$$
 Since the number of nonnegative integer $p$-tuples with summation $\le T$ is equal to the number of nonnegative integer $(p+1)$-tuples with summation exactly equal to $T\in\mathbf Z$, which is $\binom{T+n}{n}$, and since the sum below is independent of $s$, we have
\begin{align*}
a&=\sum_{||(i)_k||\le\frac{L}{d}}m^I_ki_{sk}\ge \sum_{||(i)_k||\le\frac{L}{d}-n}m^I_ki_{sk}=\dfrac{d^n}{n+1}\sum_{||(i)_k||\le\frac{L}{d}-n}\sum_{s=1}^{n+1}i_{sk}\\
&=\dfrac{d^n}{n+1}\cdot \binom{\frac{L}{d}}{n}\cdot (\dfrac{L}{d}-n)=d^n\binom{\frac{L}{d}}{n+1}\\
&=d^n\left (\frac{L^{n+1}}{d^{n+1}(n+1)!}+O(L^n)\right).
\end{align*}
Therefore,
$$ \dfrac{Lu+\epsilon'}{da}=\frac{\frac{L^{n+1}}{n!}+O(L^n)}{\frac{L^{n+1}}{(n+1)!}+O(L^n)}.$$
Then for $\epsilon >0$, we may choose $L$ big enough such that
$$ \dfrac{Lu+\epsilon'}{da}\le (n+1)+\frac{\epsilon}{2(N-n+1)}.$$
Hence (\ref{3.9}) implies that
$$ \sum_{v\in S}\sum_{i=1}^q\frac{1}{d}\lambda_{Q_i(\alpha),v}(x(\alpha))\le \left ((N-n+1)(n+1)+\frac{\epsilon}{2}\right )h(x(\alpha))+o(h(x(\alpha))).
 $$
Since the term $o(h(x(\alpha)))$ will be dominated by $\frac{\epsilon}{2}h(x(\alpha))$, the above inequality yields that
$$  \sum_{v\in S}\sum_{i=1}^q\frac{1}{d}\lambda_{Q_i(\alpha),v}(x(\alpha))\le \left ((N-n+1)(n+1)+\epsilon\right )h(x(\alpha)).$$
The theorem is proved.
\end{proof}

\begin{proof}[{\sc Proof of Main Theorem}]
Assume that
$$Q_i=\sum_{I\in\mathcal T_{d_i}}a_{i,I}x^I,\ a_{i,I}\in\mathcal R^0_\Lambda.$$
By changing the homogeneous coordinates of $\P^n(k)$ if necessary, we may assume that $a_{iI^i_0}\ne 0$ for all $i=1,...,q$, where $I^i_0=(d_i,0,...,0)$. 

Let $A$ be an infinite subset of $\Lambda$ which is coherent with respect to $\{Q_i\}_{i=1}^q$. Replacing $A$ by an infinite subset with finite complement, we may assume that $a_{iI^i_0}(\alpha)\ne 0$ for all $\alpha\in A$. By the assumption that $x$ is algebraically nondegenerate over $\mathcal R_{\{Q_i\}_{i=1}^q}$, we have for each $Q_i$,
$$ Q_i(\alpha)(x(\alpha))\ne 0 $$
for all but finite many $\alpha\in A$. Then by passing to an infinite subset with finite complement of $A$, we may assume that
$$ Q_i(\alpha)(x(\alpha))\ne 0 $$
for all $i=1,...,q$ and $\alpha\in A.$

We set
$$ \tilde Q_i=\left (\dfrac{1}{a_{iI^i_0}}Q_i\right)^{d/d_i}, $$
where $d$ is the least common multiple of $d_1,...,d_q$. Then we see that $\tilde Q_i\ (1\le i\le q)$ have the same degree $d$ and if we write
$$ \tilde Q_i=\sum_{I\in\mathcal T_d}\tilde a_{i,I}x^I $$
then $a_{i,I_0}=1$, where $I_0=(d,0,...,0)$.

We now verify that $A$ is coherent with respect to $\{\tilde Q_i\}_{i=1}^q$. Indeed, suppose that we have a non-zero polynomial $P(...,x_{1,I},...,x_{i,J},...,x_{q,K},...)$, $I,J,K\in\mathcal T_d$, which is homogeneous in $(...,x_{i,I},...)$ for each $i\in\{1,...,q\}$ and satisfies
\begin{align}\label{3.10}
P(...,\tilde a_{1,I}(\alpha),...,\tilde a_{i,J}(\alpha),...,\tilde a_{q,K}(\alpha),...)\ne 0.
\end{align} 
for some $\alpha\in A$. Then 
\begin{align}\label{3.11}
P(...,a_{1I^1_0}^{d/d_1}\tilde a_{1,I}(\alpha),...,a_{i,I^i_0}^{d/d_i}\tilde a_{iJ}(\alpha),...,a_{qI^q_0}^{d/d_q}\tilde a_{qK}(\alpha),...)\ne 0
\end{align} 
for some $\alpha\in A$. Since $A$ is coherent with respect to $\{Q_i\}_{i=1}^q$, (\ref{3.11}) holds for all $\alpha\in A$ outside a finite subset. This yields that $A$ is also coherent with respect to $\{\tilde Q_i\}_{i=1}^q$.

It is easy to see that $\{\tilde Q_i\}_{i=1}^q$ is in $N$-subgeneral position. And for an infinite subset $\mathcal B$ of $\Lambda$ we have
$$ \mathcal R_{\mathcal B,\{\tilde Q_i\}_{i=1}^q}=\mathcal R_{\mathcal B,\{Q_i\}_{i=1}^q},$$
and hence $x$ is algebraically non-degenerate over $\mathcal R_{\{\tilde Q_i\}_{i=1}^q}$. Moreover, we have
$$ h(\tilde Q_i(\alpha))=\dfrac{d}{d_i}h(Q_i(\alpha))=o(h(x(\alpha))),\ \ i=1,...,q.$$ 

Therefore, the family $\{\tilde Q_i\}_{i=1}^n$ satisfies the assumptions of Theorem \ref{thm3.2} with respect to the collection of points $x$. Applying Theorem \ref{thm3.2},  there exists an infinite index subset $B\subset A\subset\Lambda$ such that
$$\sum_{v\in S}\sum_{i=1}^q\frac{1}{d}\lambda_{\tilde Q_i(\alpha),v}(x(\alpha))\le ((N-n+1)(n+1)+\epsilon) h(x(\alpha)),$$
for all $\alpha\in B$. Since $\frac{1}{d}\lambda_{\tilde Q_i(\alpha),v}(x(\alpha))=\frac{1}{d_i}\lambda_{Q_i(\alpha),v}(x(\alpha))$, the above inequality implies that
$$\sum_{v\in S}\sum_{i=1}^q\frac{1}{d_i}\lambda_{Q_i(\alpha),v}(x(\alpha))\le ((N-n+1)(n+1)+\epsilon) h(x(\alpha)),$$
for all $\alpha\in B$. This proves the theorem.
\end{proof}

\noindent
{\bf Acknowledgements.} This research is funded by Vietnam National Foundation for Science and Technology Development (NAFOSTED) under grant number 101.04-2015.03.

\vskip0.2cm
{\footnotesize 
\noindent
{\sc Si Duc Quang}\\
$^1$ Department of Mathematics,\\
Hanoi National University of Education,\\
136-Xuan Thuy, Cau Giay, Hanoi, Vietnam.\\
$^2$ Thang Long Institute of Mathematics and Applied Sciences,\\
Nghiem Xuan Yem, Hoang Mai, HaNoi, Vietnam.\\
\textit{E-mail}: quangsd@hnue.edu.vn


\begin{thebibliography}{HD}

\bibitem{CRY} Z. Chen, M. Ru and Q. Yan, \textit{Schmidt's subspace theorem with moving hypersurfaces,} Intern. Math. Research Notices,  Int. Math. Res. Not. \textbf{15} (2015) 6305--6329.

\bibitem{CZ} P. Corvaja and U. Zannier,\textit{On a general Thue's equation,} Amer. J. Math.\textbf{126} (2004) 1033--1055.

\bibitem{DT} G. Dethloff and T. V. Tan, \textit{A second main theorem for moving hypersurface targets}, Houston J. Math. 37 (2011) 79--111.

\bibitem{LG} L. Giang, \textit{Schmidt's subspace theorem for moving hypersurface targets,} Intern. J. Number Theory \textbf{11} (2015) 139--158.

\bibitem{Os1} C. F. Osgood, \textit{A number theoretic-differential equations approach to generalizing Nevanlinna theory}, Indian J. Math. \textbf{23} (1981) 1--15.

\bibitem{Os2} C. F. Osgood, \textit{Sometimes effective Thue-Siegel-Roth-Schmidt-Nevanlinna bounds, or better,} J. Number Theory \textbf{21} (1985) 347--389.

\bibitem{Q16-1} S. D. Quang, \textit{Degeneracy second main theorems for meromorphic mappings into projective varieties with hypersurfaces}, arXiv:1610.03951 [math.CV].

\bibitem{Q16-2} S. D. Quang, \textit{Second main theorem for meromorphic mappings with moving hypersurfaces in subgeneral position}, Preprint.

\bibitem{RS} M. Ru and W. Stoll, \textit{The second main theorem for moving targets}, J. Geom. Anal.\textbf{1} (1991), 99--138.

\bibitem{RV} M. Ru and P. Vojta, \textit{Schmidt's subspace theorem with moving targets, }Invent. Math. \textbf{127} (1997) 51--65.

\bibitem{V} P. Vojta, \textit{Diophantine Approximation and Value Distribution Theory}, Lecture Notes in Mathematics, Vol. 1239 (Springer, Berlin, 1987).

\bibitem{Shi} M. Shirosaki,\textit{Another proof of the defect relation for moving targets},
T\^ohoku Math. J. {\bf 43} (1991), 355--360.

\end{thebibliography}
\end{document}